\documentclass[12pt,a4paper]{article}
\usepackage[T1]{fontenc}
\usepackage{lmodern}
\usepackage[utf8]{inputenc}
\usepackage{amsthm}
\usepackage{amsmath,amssymb}
\usepackage{graphics}
\usepackage{lscape,graphicx}
\usepackage{enumitem}
\usepackage{amsfonts}
\usepackage[T1]{fontenc}
\usepackage{longtable}
\usepackage{authblk}
\usepackage{lipsum}
\usepackage{epsfig}
\usepackage{float}
\usepackage{hyperref}
\usepackage{comment}
\usepackage[normalem]{ulem}
\usepackage{multicol}
\usepackage{fancyhdr}
\usepackage[english]{babel}
\usepackage{mathrsfs}
\usepackage{tocloft}
\usepackage{xcolor}
\usepackage{titletoc}
\usepackage{mathtools}
\usepackage[toc,page]{appendix}

\addto{\captionsenglish}{}
\makeatletter
\def\thm@space@setup{%
  \thm@preskip=1cm plus .5cm minus .5cm
  \thm@postskip=.5cm plus .6cm minus .5cm 
}

\makeatother

\newtheorem{thm}{Theorem}
\newtheorem{lma}{Lemma}

\newtheorem{prop}{Proposition}

\numberwithin{thm}{section}
\numberwithin{lma}{section}
\numberwithin{dfn}{section}
\numberwithin{cor}{section}
\numberwithin{rmk}{section}
\numberwithin{prop}{section}

\newcommand*{\thmref}[1]{Theorem~\ref{#1}}
\newcommand*{\lmaref}[1]{Lemma~\ref{#1}}

\newcommand*{\propref}[1]{Proposition~\ref{#1}}

\title{On the number of irreducible factors with a given multiplicity in function fields}

\date{}

\begin{document}

\author{Sourabhashis Das, Ertan Elma, Wentang Kuo, Yu-Ru Liu}

\newcommand{\Addresses}{{
  \bigskip
  \footnotesize

  Sourabhashis Das (Corresponding author), Department of Pure Mathematics, University of Waterloo, 200 University Avenue West, Waterloo, Ontario, Canada, N2L 3G1. \\
  Email address: \texttt{s57das@uwaterloo.ca}

  \medskip

  Ertan Elma, Department of Mathematics and Computer Science, University of Lethbridge, 4401 University Drive, Lethbridge, Alberta,  Canada, T1K 3M4. \\
  Email address: \texttt{ertan.elma@uleth.ca}

  \medskip

  Wentang Kuo, Department of Pure Mathematics, University of Waterloo, 200 University Avenue West, Waterloo, Ontario, Canada, N2L 3G1. \\
  Email address: \texttt{wtkuo@uwaterloo.ca}
  
  \medskip

  Yu-Ru Liu, Department of Pure Mathematics, University of Waterloo, 200 University Avenue West, Waterloo, Ontario, Canada, N2L 3G1. \\
  Email address: \texttt{yrliu@uwaterloo.ca}

}}

%
%
%

\maketitle 

\begin{abstract}
Let $k \geq 1$ be a natural number and $f \in \mathbb{F}_q[t]$ be a monic polynomial. Let $\omega_k(f)$ denote the number of distinct monic irreducible factors of $f$ with multiplicity $k$. We obtain asymptotic estimates for the first and the second moments of $\omega_k(f)$ with $k \geq 1$. Moreover, we prove that the function $\omega_1(f)$ has  normal order $\log (\text{deg}(f))$ and also satisfies the Erd\H{o}s-Kac Theorem. Finally, we prove that the functions $\omega_k(f)$ with $k \geq 2$ do not have normal order.
\end{abstract}

\section{Introduction}

In\footnotetext{\textbf{2020 Mathematics Subject Classification: 11T06, 11N37, 11N56.}} 1940,\footnotetext{\textbf{Keywords: monic irreducible factors, normal order, Erd\H{o}s-Kac theorem.}} Erd\H{o}s and Kac,\footnotetext{The research of W. Kuo and Y.-R. Liu are supported by NSERC discovery grants. The authors also want to thank M. Lal\'in and Z. Zhang for the helpful discussion.} \cite{ErdosKac}, proved the remarkable result that the number of distinct prime factors, $\omega(n)$, of a natural number $n$ obeys the Gaussian law in the sense that
\begin{align}\label{ErdosKacResult}
	\lim_{x\rightarrow \infty}\frac{1}{x}\left|\lbrace  3\leqslant n \leqslant x : a\leqslant \frac{\omega(n)-\log\log n}{\sqrt{\log\log n}}\leqslant b   \rbrace\right|=
	\Phi(b) - \Phi(a)
\end{align}
where
\begin{equation}\label{Nadef}
		\Phi(t) = \frac{1}{\sqrt{2 \pi}} \int_{-\infty}^t e^{-\frac{u^2}{2}} du.
	\end{equation}
	Various approaches to the Erd\H{o}s-Kac theorem have been pursued, for example, see \cite{pb}, \cite{gs}, \cite{hh1}, \cite{hh2}, \cite{hh3}, \cite{jala2}.
	
In \cite{el}, the second and the fourth authors considered a refinement of the function $\omega(n)$ in the following way. Let $\omega_k(n)$ be the number of distinct prime factors of $n$ with multiplicity $k\geqslant 1$. In \cite{el}, the first and the second moments of $\omega_k(n)$ are obtained. Moreover, they showed that $\omega_1(n)$ has normal order\footnote{Let $f,F:\mathbb{N}\rightarrow \mathbb{R}_{\geqslant 0}$ be two functions such that $F$ is non-decreasing. Then $f(n)$ is said to have \textit{normal order} $F(n)$ if for any $\epsilon>0$, the number of $n\leqslant x$ that do not satisfy the inequality
	\begin{align*}
		(1-\epsilon)F(n) \leq f(n) \leq (1+\epsilon)F(n)
	\end{align*}
	is $o(x)$ as $x\rightarrow \infty$.} $\log\log n$, that $\omega_1(n)$ satisfies the Gaussian law given in \eqref{ErdosKacResult}, and that the function $\omega_k(n)$ with $k\geqslant 2$ does not have normal order $F(n)$ for any nondecreasing nonnegative function $F(n)$.

In this work, we consider the refined version of the function field analogue of $\omega(n)$. Let $\mathbb{F}_q[t]$ be the ring of polynomials in the variable $t$ with coefficients from the finite field $\mathbb{F}_q$ with $q$ elements. Let $M$ be the set of monic polynomials in $\mathbb{F}_q[t]$ with degree $\text{deg}(f)\geqslant 1$ and $P$ be the set of monic irreducible polynomials in $\mathbb{F}_q[t]$ which is the analogue of prime numbers in the ring $\mathbb{Z}$. For $f\in\mathbb{F}_q[t]$, define $\left|f\right|:=q^{\deg(f)}$.


For a monic polynomial $f\in M$ and a monic irreducible $l\in P$, let $\nu_l(f)$ be the multiplicity of $l$ in the unique factorization of $f$, that is, $\nu_l(f)$ is the largest nonnegative integer such that $l^{{\nu_l(f)}}\mid f$ but $l^{\nu_l(f)+1}\nmid f$. Let 
\begin{align*}
	\omega(f):=\sum_{l\mid f}1
\end{align*}
which counts the number of distinct monic irreducible factors $l\in P$ of $f\in M$, and 
\begin{align*}
	\omega_k(f):=\sum_{\substack{l\mid f\\\nu_l(f)=k}}1, \quad (k\geqslant 1)
\end{align*}
which counts the number of distinct monic irreducible factors of $f$ with a given multiplicity $k\geqslant 1$. Note that the set of functions $\omega_k(f)$ with $k\geqslant 1$ is a refinement of $\omega(f)$ in the sense that 
\begin{align*}
	\omega(f)=\sum_{k\geqslant 1} \omega_k(f)
\end{align*}
for all $f\in M$. 

For a natural number $n\geqslant 1$, let $M_n$ be the set of monic polynomials in $\mathbb{F}_q[t]$ with degree $n$. By the work of A. Knopfmacher and J. Knopfmacher, \cite[Theorems 1 and 4 or the comment on Page 111]{kk},  
the mean value of $\omega(f)$ over all polynomials in $M_n$ is given by
\begin{equation}\label{first_moment_omega}
	\sum_{f \in M_n} \omega(f) = q^n(\log n) + A_1 q^n + O \left( \frac{q^n}{n} \right)
\end{equation}
with
\begin{equation}\label{defA1}
	A_1 = \gamma + c_1,
\end{equation}
where $\gamma$ is the Euler-Mascheroni constant, and $c_1 = \sum_{r=1}^{\infty} (\pi_q(r) - q^r/r) q^{-r}$ with $\pi_q (r)$ denoting the number of monic irreducible polynomials $l\in P$ with degree $\text{deg}(l)=r$. 

First we obtain refinements of the result in (\ref{first_moment_omega}) by estimating the first moment of the functions $\omega_k(\cdot)$ with $k\geqslant1$.
\begin{thm}\label{mainresult_first_moment}
For an integer $m \geq 2$, let
\begin{equation}\label{Lk}
	L(m) := \sum_{l \in P} \frac{1}{|l|^m}.
\end{equation}
We have 
$$\sum_{f \in M_n} \omega_1(f) = q^n (\log n) + q^n (A_1 - L(2)) + O \left( \frac{q^n}{n} \right)$$
where $A_1$ is defined in \eqref{defA1}. Moreover for $k \geq 2$, as $n \rightarrow \infty$, we have
$$\sum_{f \in M_n} \omega_k(f) = (L(k) - L(k+1))  q^n + O \left(n \cdot q^{\frac{n}{k}} \right).$$	
\end{thm}
For the second moments of $\omega_k(\cdot)$, we obtain the following result.
\begin{thm}\label{mainresult_second_moment}
	Let $A_1$ and $L(m)$ be defined as in \eqref{defA1} and \eqref{Lk}, respectively. As $n \rightarrow \infty$, we have
	$$\sum_{f \in M_n} \omega_1^2(f) = q^n(\log n)^2 + c_2 q^n (\log n) + c_{3} q^n +  O  \left( \frac{q^n (\log n)}{n} \right),$$
	where
	\begin{equation*}
		c_2 := 1 + 2 A_1 - 2 L(2) 
	\end{equation*}
	and
	\begin{equation}\label{c3}
		c_{3} := (A_1 - 2 L(2))(A_1 + 1) - \frac{\pi^2}{6} + (L(2))^2 + 2 L(3) - L(4).
	\end{equation}
	Moreover, for $k\geqslant 2$, we have
	$$\sum_{f \in M_n} \omega_k^2(f) = c_k' q^n +  O \left(n \cdot q^{\frac{n}{k}} \right),$$
	where
	\begin{equation*}
		c_{k}' := (L(k) - L(k+1)) (L(k) - L(k+1) + 1) - L(2k) + 2 L(2k+1) - L(2k+2).
	\end{equation*}
\end{thm}
Note that the error terms for $\omega_k$ with $k \geq 2$ as in \thmref{mainresult_first_moment} and \thmref{mainresult_second_moment} improves upon their respective error terms in the integer case (see \cite[Theorem 1.1 and Theorem 1.2]{el}) by at least a factor of $\frac{(k-1)^2}{k(3k-1)}$ in the exponent.

Let $g, G : M \rightarrow \mathbb{R}_{\geq 0}$ be two functions. We say $G(f)$ is non-decreasing if $G(f) \geq G(h)$ for all $f,h$ with $\textnormal{deg}(f) \geq \textnormal{deg}(h)$. Then $g(f)$ is said to have normal order $G(f)$ for a non-decreasing function $G(f)$ if for any $\epsilon >0$, the number of polynomials $f$ with degree $n$ that do not satisfy the inequality
	$$(1-\epsilon) G(f) \leq g(f) \leq (1+\epsilon) G(f)$$
	is $o(q^n)$ as $n \rightarrow \infty$.
	
The first and the second moment estimates in \thmref{mainresult_first_moment} and \thmref{mainresult_second_moment} allow us to obtain the normal order of $\omega_1(\cdot)$.
\begin{thm}\label{mainresult_normal_order_k=1}
Let $c_3$ be defined as in \eqref{c3}. As $n \rightarrow \infty$, we have
$$\sum_{f \in M_n} (\omega_1(f) - \log n)^2 = q^n \log n + c_3 q^n + O \left( \frac{q^n \log n}{n} \right).$$
Let $\epsilon' \in (0,1/2)$. Then the number of monic polynomials $f$ of degree $n$ such that
$$\frac{|\omega_1(f) - \log n|}{\sqrt{\log n}} \geq (\log n)^{\epsilon'}$$
is $o(q^n)$ as $n \rightarrow \infty$ and thus $\omega_1(f)$ has normal order $\log\left(\deg(f)\right)$. 	
\end{thm}
However, unlike the function $\omega_1(\cdot)$, we prove that the functions $\omega_k(\cdot)$ with $k\geqslant2$ do not have normal order.
\begin{thm}\label{mainresult_no_normal_order_k>1}
	Let $k \geq 2$ be a fixed integer. Then the function $\omega_k(f)$ does not have normal order $G(f)$ for any non-decreasing function $G: M \rightarrow \mathbb{R}_{\geq 0}$.
\end{thm}
In \cite{wz}, Zhang proved a function field analogue  of the Erd\H{o}s-Kac Theorem that
\begin{equation}\label{kacerdosver1}
\lim_{n \rightarrow \infty} \frac{\left| \left\{ f \ \Big| \ f \in M_m, \ m \leq n, \  \frac{\omega(f) - \log m}{\sqrt{\log m}} \leq a \right\} \right|}{| \{ f \ | \ f \in M_m, \ m \leq n \} |} = \Phi(a),
\end{equation}
	where $|\{ \cdot \}|$ denotes the cardinality of the corresponding set. A generalization of this result can be found in the work of the fourth author \cite{liu} and in the work of Lal\'in and Zhang \cite{lz}. 
	
In our final main result, we prove that the function $\omega_1(\cdot)$ obeys the same Gaussian distribution as in the the Erd\H{o}s-Kac Theorem.
	\begin{thm}\label{mainresult_Erdos_Kac_k=1}
	Let $a \in \mathbb{R}$. Then
	$$\lim_{n \rightarrow \infty} \frac{1}{q^n} \left| \left\{ f \in M_n \ \Big| \ \frac{\omega_1(f) - \log n}{\sqrt{\log n}} \leq a \right\} \right| = \Phi(a),$$
	where $\Phi(a)$ is defined in \eqref{Nadef}.
\end{thm}
For a monic polynomial $f\in M$, let $\Omega(f)$ be the number of monic irreducible factors of $f$. Then $\Omega(f)$ satisfies the Erd\H{o}s-Kac Theorem (see \cite[Remark, Page 605]{liu}). Let $\Omega_k(f)$ be the number of monic irreducible factors of $f$ with a given multiplicity $k\geqslant 1$. Note that $\Omega_k(f) = k \cdot \omega_k(f)$ and 
$\Omega(f)=\sum_{k\geqslant 1} \Omega_k(f)$ for all $f\in M$. We can deduce similar results for $\Omega_k(f)$ as our results above. In particular, we can prove that $\Omega_1(f)$ has  normal order $\log (\text{deg}(f))$ and also satisfies the Erd\H{o}s-Kac Theorem. We can also prove that the functions $\Omega_k(f)$ with $k \geq 2$ do not have normal order $G(f)$ for any non-decreasing function $G: M \rightarrow \mathbb{R}_{\geq 0}$.   

\section{Lemmata}
In this section, we prove some lemmata that are needed in the proof of \thmref{mainresult_first_moment} and \thmref{mainresult_second_moment}. We start by recalling the analogue of the prime number theorem for polynomials. We recall that $\pi_q(n)$ denotes the number of monic irreducible polynomials $l\in P$ with degree $\text{deg}(l)=n$. 

\begin{lma}\label{PNT}\cite[Theorem 2.2]{mr}(The prime number theorem for polynomials)
	For $n \geq 1$, we have
	$$\pi_q(n) = \frac{q^n}{n} + O \left( \frac{q^{\frac{n}{2}}}{n} \right).$$
\end{lma}

%
\begin{lma}\label{kkthm1}
	For $R>0$, let $F(z)$ be a meromorphic function on $|z| \le R$ with at most finitely many poles in the open disk $|z| < R$ and analytic on the circle $|z| =R$. Suppose that $F(z)$ is analytic at $z=0$ with the power series representation
	\begin{align*}
		F(z)=\sum_{n=0}^{\infty}c_nz^n
	\end{align*}
in the open disk $\left|z\right|<r$ for some $r>0$. As $N\rightarrow \infty$, we have 
\begin{align*}
			c_N=-\sum{\vphantom{\sum}}' \textnormal{Res} \left(\frac{F(z)}{z^{N+1}}\right)+O\left(R^{-N}\max_{\left|z\right|=R}\left|F(z)\right|\right)
	\end{align*}
where the implied constant is absolute and the sum calculates the residues of $F(z)/z^{N+1}$ at the poles of $F(z)/z^{N+1}$ in $\left|z\right|<R$, except the one that comes from a possible pole at $z=0$. 
\end{lma}

\begin{proof}
	Applying the residue theorem to the integral $	\frac{1}{2\pi i}\int_{\left|z\right|=R}\frac{F(z)}{z^{N+1}}\, dz$ and using the power series of $F(z)$ around $z=0$, we obtain the desired result.
\end{proof}

\begin{lma}\label{Lemma_Mertens}(Mertens' theorem for $\mathbb{F}_q[x]$) Let $A_1$ be defined as in \eqref{defA1}. As $n \rightarrow \infty$, we have
	$$\sum_{\substack{l \in P \\ \textnormal{deg}(l) \leq n}} \frac{1}{|l|} = \log n + A_1 + O \left( \frac{1}{n} \right).$$
\end{lma}
\begin{proof}
	By Lemma \ref{PNT}, we have
	$$\pi_q(m) = \frac{q^m}{m} + r_m$$
	where $r_m \ll \frac{q^{m/2}}{m}$ as $m \rightarrow\infty$. Thus
	$$\sum_{\substack{l \in P \\ \textnormal{deg}(l) \leq n}} \frac{1}{|l|} = \sum_{m=1}^n q^{-m} \pi_q(m) = \sum_{m=1}^n \frac{1}{m} + \sum_{m=1}^n r_m q^{-m}.$$
	By \cite[Theorem 1]{bw}, we have 
	\begin{align}\label{sum1/k}
		\sum_{m=1}^n \frac{1}{m} = \log n + \gamma + O \left( \frac{1}{n} \right).
	\end{align}
	Also, we have 
	\begin{align*}
		\sum_{m=1}^n r_m q^{-m}= c_1 - \sum_{m=n+1}^\infty r_m q^{-m}
		&=c_1+O \left( \sum_{m = n+1}^\infty \frac{q^{-m/2}}{m} \right) =c_1+O\left(\frac{q^{-n/2}}{n}\right).
	\end{align*}
	Hence the desired result follows.
\end{proof}

\begin{lma}\label{lma1}
	Let $g : P \rightarrow \mathbb{C}$ be a function such that $|g(l)| \leq 1$.
	Let $k \in \mathbb{N}$. Let $s = \sigma + it, \ \sigma, t \in \mathbb{R}$. Define 
	\begin{equation*}
		B_{g,k}(s) := \sum_{l \in P} \frac{1}{|l|^{ks} (|l|^s - g(l))}.
	\end{equation*}
	Then $B_{g,k}(s)$ is absolutely convergent for $\sigma > \frac{1}{k+1}$. Moreover, we have
	\begin{equation*}
		B_{g,k}(s) = \sum_{f \in M} \frac{b_{g,k}(f)}{|f|^s},
	\end{equation*}
	where
	\begin{equation*}
		b_{g,k}(f) = \begin{cases}
			g(l)^{\alpha - (k+1)} & \text{if } f  = l^\alpha \textit{ for some } l\in P \textit{ and }  \alpha \geq k+1, \\
			0 & \text{otherwise}.
		\end{cases}
	\end{equation*}
\end{lma}
\begin{proof}
	For $\sigma>\frac{1}{k+1}$, we have 
	\begin{align*}
		\sum_{l \in P} \frac{1}{|l|^{ks} (|l|^s - g(l))}\ll \sum_{l\in P}\frac{1}{\left|l\right|^{(k+1)\sigma}}\ll \sum_{n\geqslant 1}\frac{q^{n}}{q^{n(k+1)\sigma}}\ll 1,
	\end{align*}
	and thus the series $B_{g,k}(s)$ is absolutely convergent for $\sigma>\frac{1}{k+1}$. Since $\left|g(l)\right|\leqslant 1$ for all $l\in P$ and $\left|\frac{g(l)}{\left|l\right|^s}\right|<1$ for $\sigma>\frac{1}{k+1}$, we have 
	\begin{equation*}
		B_{g,k}(s) = \sum_{l \in P} \frac{1}{|l|^{(k+1)s}} \left( 1 + \frac{g(l)}{|l|^s} + \left( \frac{g(l)}{|l|^s} \right)^2 + \cdots \right) = \sum_{f \in M} \frac{b_{g,k}(f)}{|f|^s}
	\end{equation*}
		which completes the proof.
\end{proof}

\begin{prop}\label{prop1}
Let $g : P \rightarrow \mathbb{C}$ be a function such that $|g(l)| \leq 1$.
Let $k \in \mathbb{N}$. Define
	$$a_{g,k}(f) := \sum_{\overset{l \in P, \ l|f}{\nu_l(f) \geq k+1}}\big (1 + g(l) + (g(l))^2 + \cdots + (g(l))^{\nu_l(f) - (k+1)} \big),$$
	with the convention that the empty sum is taken to be zero.
	Define 
	$$C_{g,k} := \sum_{l \in P} \frac{1}{|l|^k (|l| - g(l))}.$$
	Then, we have
	$$\sum_{f \in M_n} a_{g,k}(f) = C_{g,k} \ q^n + O \left( n \cdot q^{\frac{n}{k+1}} \right),$$
	where the implied constant is absolute. 
\end{prop}
\begin{proof}
	Let $s = \sigma + it, \ \sigma, t \in \mathbb{R}$. We define
	$$A_{g,k}(s) := \sum_{f \in M} \frac{a_{g,k}(f)}{|f|^s} = \sum_{n \geq 1} \left( \sum_{f \in M_n} a_{g,k}(f) \right) \  q^{-n s}.$$
	Note that, we will show later in this proof that $A_{g,k}(s)$ is absolutely convergent for $\sigma > 1$. 
	
	 The zeta function $\zeta_q(s)$ in $\mathbb{F}_q[t]$ is defined by
	 \begin{align*}
	 	\zeta_q(s) = \sum_{f \in M} \frac{1}{|f|^s} = \frac{1}{1 - q^{1-s}}\quad (\sigma > 1).
	 \end{align*}
Thus, we have
	$$ \zeta_q(s) B_{g,k}(s) = \sum_{f \in M} \frac{1}{|f|^s} \cdot \sum_{f \in M} \frac{b_{g,k}(f)}{|f|^s} = \sum_{f \in M} \frac{\sum_{d | f} b_{g,k}(d)}{|f|^s}.$$
	Notice that
	$$\sum_{d|f} b_{g,k}(d) = \sum_{\substack{l \in P, \ l|f \\ \nu_l(f) \geq k+1}} \left( \sum_{j = k+1}^{\nu_l(f)} (g(l))^{j - (k+1)} \right) = a_{g,k}(f).$$
	Since $\zeta_q(s)$ and $B_{g,k}(s)$ converge absolutely for $\sigma > 1$ and $\sigma > \frac{1}{k+1}$ respectively, we conclude that  
	$$\sum_{f \in M} \frac{a_{g,k}(f)}{|f|^s} = A_{g,k}(s) = \zeta_q(s) B_{g,k}(s)$$
	converges absolutely in the region $\sigma > 1$. Additionally, $A_{g,k}(s)$ admits a meromorphic continuation up to $\sigma > \frac{1}{k+1}$ with simple poles at $1 + i \frac{2 \pi k}{\log q}$ where $k \in \mathbb{Z}$. Making the change of variable $u = q^{-s}$, we introduce $\mathcal{B}_{g,k}(u) \ (= B_{g,k}(s))$ and $\mathcal{A}_{g,k}(u) \ (= A_{g,k}(s))$ as  
	\begin{equation*}
		\mathcal{B}_{g,k}(u) = \sum_{l \in P} \frac{u^{(k+1)\text{deg}(l)}}{1 - g(l) u^{\text{deg}(l)}} \quad \text{and} \quad \mathcal{A}_{g,k}(u) = \sum_{f \in M} a_{g,k}(f) u^{\text{deg}(f)}.
	\end{equation*}
	Notice that $\mathcal{A}_{g,k}(u)$ admits a meromorphic continuation to $|u| < q^{-\frac{1}{k+1}}$ with only a simple pole at $q^{-1}$. 
	Moreover as $\zeta_q(s) = \mathcal{Z}_q(u) = \frac{1}{1 - qu}$, we have
	$$\mathcal{A}_{g,k}(u) = \mathcal{Z}_q(u) \mathcal{B}_{g,k}(u) = \frac{\mathcal{B}_{g,k}(u)}{1-qu}.$$
	Since $\mathcal{A}_{g,k}(u)$ is analytic at $u =0$, we use the power series of $\mathcal{A}_{g,k}(u)$ around $u=0$ (in a small disk $|u|<r < q^{-1}$) in \lmaref{kkthm1} to complete the proof. In order to do that, we write $\mathcal{A}_{g,k}(u)$ as 
	$$\mathcal{A}_{g,k}(u) = \sum_{n=0}^\infty  \left( \sum_{f \in M_n} a_{g,k}(f) \right) u^n$$
	and for any $\epsilon \in (0,1/4]$, we choose $R =  q^{- \left(\frac{1}{k+1} + \epsilon \right)}  > q^{-1}$ to satisfy the hypothesis of the lemma for the function $\mathcal{A}_{g,k}(u) = \frac{\mathcal{B}_{g,k}(u)}{1-qu}$. Thus applying the lemma, we obtain
	\begin{equation}\label{impo0}
		\sum_{f \in M_n} a_{g,k}(f) = - \text{Res}_{u = q^{-1}} \frac{\mathcal{B}_{g,k}(u)}{u^{n+1} (1 - qu)} + O \left( R^{-n} \max_{\left|u\right|=R}\left| \frac{\mathcal{B}_{g,k}(u)}{1-qu}\right| \right),
	\end{equation}
	where the implied constant is absolute. Note that $\mathcal{B}_{g,k}(q^{-1}) = C_{g,k}$ and thus 
	\begin{equation}\label{residue0}
		- \text{Res}_{u = q^{-1}} \frac{\mathcal{B}_{g,k}(u)}{u^{n+1} (1 - qu)} = C_{g,k} \ q^n.
	\end{equation}
	Moreover, on $|u| = R = q^{- \left(\frac{1}{k+1} + \epsilon \right)}$, we have $\frac{1}{1-qu}\ll  1$ and
	$\mathcal{B}_{g,k}(u) \ll \frac{1}{q^{(k+1)\epsilon}-1}$. Thus, the error term in \eqref{impo0} is bounded by
\begin{align*}
	\ll q^{\frac{n}{k+1}}\frac{q^{n\epsilon}}{q^{(k+1)\epsilon}-1}.
\end{align*}
We choose $\epsilon$ such that $q^{(k+1)\epsilon}=\frac{n+1}{n}$. Then $\epsilon=\frac{1}{(k+1)\log q}\log \left(\frac{n+1}{n}\right)\leqslant \frac{1}{4}$ for $k\geqslant 1,\, q\geqslant 2, n\geqslant 3$, and we have 
\begin{align*}
	\frac{q^{n\epsilon}}{q^{(k+1)\epsilon}-1}&=q^{n \frac{1}{(k+1)\log q}\log \left(\frac{n+1}{n}\right)}\frac{1}{\frac{n+1}{n}-1}
	\\&=n\exp\left(\frac{1}{k+1}n\log\left(1+\frac{1}{n}\right)\right)
	\\&\ll n,
\end{align*}
where the implied constant is absolute. Thus the error term in \eqref{impo0} is bounded by $ nq^{\frac{n}{k+1}}$, and this together with \eqref{residue0} completes the proof.	
\end{proof}
\begin{lma}\label{lemmakn-k}
	As $n \rightarrow \infty$, we have
	$$\sum_{m=1}^{n-1} \frac{1}{m (n-m)} = \frac{2 \log n}{n} + \frac{2 \gamma}{n} + O \left( \frac{1}{n^2} \right).$$
\end{lma}
\begin{proof}
	Note that, using symmetry and \eqref{sum1/k}, we have
	\begin{align*}
		\sum_{m=1}^{n-1} \frac{1}{m(n-m)} & = \frac{1}{n} \sum_{m=1}^{n-1} \left( \frac{1}{n-m} + \frac{1}{m} \right)  \notag \\
		& = \frac{2}{n} \sum_{m=1}^{n-1} \frac{1}{m} \notag \\
		& = \frac{2 \log(n-1)}{n} + \frac{2 \gamma}{n} + O \left( \frac{1}{n (n-1)} \right).
	\end{align*}
	We complete the proof by noticing that
	as $n \rightarrow \infty$,
	\begin{equation*}
		\log(n-1) = \log n + \log \left( 1 - \frac{1}{n}\right) = \log n + O \left( \frac{1}{n} \right).
	\end{equation*}
\end{proof}
\begin{lma}\label{lemmak2n-k}
	As $n \rightarrow \infty$, we have
	$$\sum_{m=1}^{n-1} \frac{1}{m^2(n-m)} = \frac{\pi^2}{6n} + \frac{2 \log n}{n^2} + O \left( \frac{1}{n^2} \right).$$
\end{lma}
\begin{proof}
	Notice that
	$$\sum_{m=1}^{n-1} \frac{1}{m^2(n-m)}  = \frac{1}{n} \left( \sum_{m=1}^{n-1} \frac{1}{m(n-m)} + \sum_{m=1}^{n-1} \frac{1}{m^2} \right).$$
	Using
	\begin{equation*}
		\sum_{m=1}^{n-1} \frac{1}{m^2} = \sum_{m=1}^{\infty} \frac{1}{m^2} - \sum_{m = n}^{\infty} \frac{1}{m^2} = \frac{\pi^2}{6} + O \left( \int_n^\infty \frac{dt}{t^2} \right) = \frac{\pi^2}{6} + O \left( \frac{1}{n} \right)
	\end{equation*}
	together with \lmaref{lemmakn-k} completes the proof.
\end{proof}
\begin{lma}\label{sumqk/2}
	As $n \rightarrow \infty$, we have
	\begin{equation*}
		\sum_{m=1}^{n-1} \frac{q^{-m/2}}{m(n-m)} = \frac{\log \left( 1 + \frac{1}{q^{1/2} - 1} \right)}{n} + 
		O \left( \frac{1}{n^2}\right).
	\end{equation*}
\end{lma}
\begin{proof}
	Notice that
	\begin{align*}
	\sum_{m=1}^{n-1} \frac{q^{-m/2}}{m(n-m)} & = \frac{1}{n} \left( \sum_{m=1}^{n-1} \frac{q^{-m/2}}{n - m} + \sum_{m=1}^{n-1} \frac{q^{-m/2}}{m} \right) \\
	& = \frac{1}{n} \sum_{m=1}^{n-1} \frac{q^{-m/2}}{n - m} + \frac{1}{n}  \sum_{m=1}^{\infty} \frac{q^{-m/2}}{m} + O \left( \frac{q^{-n/2}}{n^2} \right).
	\end{align*}
	For the second sum on the right hand side above, by using
	\begin{align*}
		 - \log (1- x) = \sum_{m=1}^\infty \frac{x^m}{m}\quad\quad \left(\left|x\right|<1\right)
	\end{align*}
	with  $x = q^{-1/2} < 1$, we obtain
	$$ \frac{1}{n} \sum_{m=1}^{\infty} \frac{q^{-m/2}}{m} =  \frac{- \log \left( 1 - q^{-1/2} \right)}{n} = \frac{\log \left( 1 + \frac{1}{q^{1/2}-1}\right)}{n}.$$
	Thus we complete the proof by showing
	$$ \sum_{m=1}^{n-1} \frac{q^{-m/2}}{n - m} = O \left( \frac{1}{n} \right)$$
	We write
	$$\sum_{m=1}^{n-1} \frac{q^{-m/2}}{n - m} = \sum_{m=1}^{n-1} \frac{q^{-(n-m)/2}}{m} = q^{-n/2} H_n(q^{1/2}),$$
	where
	$$H_n(x) = \sum_{m=1}^{n-1} \frac{x^m}{m}.$$
	We then notice that
	$$\left( 1 - \frac{1}{x} \right)H_n(x) = \sum_{m=1}^{n-1} \frac{x^m}{m} - \sum_{m=0}^{n-2} \frac{x^m}{m+1} = \frac{x^{n-1}}{n-1} - 1 + \sum_{m=1}^{n-2} \frac{x^m}{m(m+1)}$$
	which gives
	\begin{align*}
		x^{-n} H_n(x) & = \frac{1}{(n-1)(x-1)} - \frac{x^{-(n-1)}}{x-1} + \frac{x}{x-1} \sum_{m=1}^{n-2} \frac{x^{-(n-m)}}{m(m+1)}.
	\end{align*}
	Taking $x = q^{1/2}$, we obtain
	\begin{align*}
		x^{-n} H_n(x)= O \left( \frac{1}{n} \right) + O \left( \sum_{m=1}^{n-1} \frac{1}{m^2(n-m)} \right)
	\end{align*}
where the implied constants are absolute	and using \lmaref{lemmak2n-k} completes the proof. 

\end{proof}

\section{Proof of Theorem \ref{mainresult_first_moment}}

To prove \thmref{mainresult_first_moment}, we take $g(l) = -1$ for all $l\in P$ in \propref{prop1}. Then we have
	$$a_k(f) := a_{g,k}(f) = \sum_{\substack{l \in P, \ l|f \\ \nu_l(f) \geq k+1 \\ \nu_l(f) - k \text{ odd}}} 1,$$ 
	and note that
	\begin{equation}\label{uselate}
		\omega_1(f) = \omega(f) - a_1(f) - a_2(f).
	\end{equation} 
	Moreover, by the definition of $C_k := C_{g,k}$ in \propref{prop1}, we have 
	$$C_1 = \sum_{l \in P} \frac{1}{|l|(\left|l\right|+1)} \quad \textnormal{ and } \quad C_2 = \sum_{l \in P} \frac{1}{|l|^2(\left|l\right|+1)},$$ 
	and
	\begin{equation*}
		C_1 + C_2 = L(2).
	\end{equation*}
	Thus, by (\ref{first_moment_omega}) and \propref{prop1}, we have
	\begin{align*}
		\sum_{f \in M_n} \omega_1(f) & =  \sum_{f \in M_n} (\omega(f) - a_1(f) - a_2(f)) \\
		& = q^n \Big( \log n + A_1 + O \Big( \frac{1}{n} \Big) \Big) -  \left( C_1 + C_2 \right) q^n + O \left( n \cdot q^{\frac{n}{2}} \right) 
		\\
		& = q^n (\log n) + q^n (A_1 - L(2)) + O \Big( \frac{q^n}{n}\Big).
	\end{align*}
We now consider the cases $k \geqslant 2$. Since
	$$\omega_k(f) = a_{k-1}(f) - a_{k+1}(f)$$ 
	and
	\begin{equation*}
		C_{k-1} - C_{k+1} = \sum_{l \in P} \left( \frac{1}{|l|^{k-1}(\left|l\right|+1)} - \frac{1}{|l|^{k+1}(\left|l\right|+1)} \right) = L(k) - L(k+1),
	\end{equation*}
by \propref{prop1}, the result follows. 
\qed
\section{Proof of Theorem \ref{mainresult_second_moment}}
To prove \thmref{mainresult_second_moment}, we start with the first moment of $\omega_1$. For $l \in P$ and $f \in M$, and a positive integer $k$, we use $l^k || f$ to denote that $l^k | f$ and $l^{k+1} \nmid f$. We have
\begin{align}\label{Expanding_square}
	\sum_{f \in M_n} \omega_1^2(f) = \sum_{f \in M_n} \left( \sum_{\substack{l \in P \\ l || f}} 1 \right)^2 = \sum_{f \in M_n}  \sum_{\substack{l,h \in P \\ l||f, \ h||f }}  1 = \sum_{f \in M_n} \omega_1(f) + \sum_{f \in M_n}  \sum_{\substack{l,h \in P, \ l \neq h \\ l ||f , \ h || f}}  1 .
\end{align}
The first sum on the right-hand side above can be estimated by Theorem \ref{mainresult_first_moment}. For the second sum on the right-hand side above, we have 
\begin{equation*}\label{sec2}
	\sum_{f \in M_n}  \sum_{\substack{l,h \in P, \ l \neq h \\ l ||f , \ h || f}}  1  = \sum_{\substack{l,h \in P, \ l \neq h \\ \textnormal{deg}(l), \textnormal{deg}(h) \leq n}}  \sum_{\substack{f \in M_n \\ l || f , \ h || f}} 1, 
\end{equation*}
and the inner sum on the right-hand side can be written as 
\begin{equation*}
	\sum_{\substack{f \in M_n \\ l || f, \ h || f}} 1  = \sum_{\substack{f \in M_n  \\ (lh) | f }} 1 \ - \sum_{\substack{f \in M_n \\ (l^2 h) | f }} 1 \ - \sum_{\substack{f \in M_n  \\ (l h^2) | f}} 1 + \sum_{\substack{f \in M_n  \\ (l^2 h^2) | f}} 1. 
\end{equation*}
For $w \in M$ with $\deg(w) \leqslant n$, the number of multiples of $w$ in $M_n$ is $\frac{q^n}{\left|w\right|}$, and thus we have 
\begin{equation*}
	\sum_{\substack{l,h \in P, \ l \neq h \\ \textnormal{deg}(l), \textnormal{deg}(h) \leq n}}  \sum_{\substack{f \in M_n \\ l || f , \ h || f}} 1  = q^n (S_1(n) - 2 S_2(n) + S_3(n)), 
\end{equation*}
where 
\begin{equation*}
	S_1(n) =  \sum_{\substack{l,h \in P, \ l \neq h \\ \textnormal{deg}(lh) \leq n}} \frac{1}{|l||h|},
\end{equation*}
\begin{equation*}
	S_2(n) =  \sum_{\substack{l,h \in P, \ l \neq h \\ \textnormal{deg}(l^2 h) \leq n}} \frac{1}{|l|^2 |h|},
\end{equation*}
\begin{equation*}
	S_3(n) =  \sum_{\substack{l,h \in P, \ l \neq h \\ \textnormal{deg}(l^2 h^2) \leq n}} \frac{1}{|l|^2|h|^2}.
\end{equation*}

Now we show that 
\begin{align}\label{S_1_result}
		S_1(n) = \log^2 n + 2 A_1 \log n+ A_1^2 - \frac{\pi^2}{6} - L(2)  + O \left( \frac{\log n}{n}\right).
\end{align}
We have 
\begin{align}\label{Eq1_for_S_1}
	S_1(n) = \sum_{\substack{l,h \in P \\ \textnormal{deg}(ls) \leq n}} \frac{1}{|l||h|} - \sum_{\substack{l\in P \\ \textnormal{deg}(l) \leq  n/2}} \frac{1}{|l|^2}.
\end{align}
For the last term above, we have
$$\sum_{\substack{l \in P \\ \textnormal{deg}(l) \leq n/2}} \frac{1}{|l|^2} = \sum_{l \in P} \frac{1}{|l|^2} +  O \left(\sum_{m>n/2}\frac{q^{-m}}{m}\right)
= L(2) + O \left( \frac{q^{- n/2}}{n} \right)$$
since 
$\pi_q(m) = q^m/m + O(q^{m/2}/m)$ for $m \geq 1$. Moreover, we have
	\begin{equation*}
		\sum_{\substack{l \in P \\ \textnormal{deg}(l) \leq n-1}} \frac{1}{|l| (n - \textnormal{deg}(l))} =  \sum_{m=1}^{n-1} \frac{1}{m(n-m)} + O \left( \sum_{m=1}^{n-1} \frac{q^{-m/2}}{m(n-m)} \right) = \frac{2 \log n}{n} + O \left( \frac{1}{n} \right)
	\end{equation*}
	by \lmaref{lemmakn-k} and \lmaref{sumqk/2}. Thus for the first term on the right-hand side of (\ref{Eq1_for_S_1}), by Lemma \ref{Lemma_Mertens}, we have 
\begin{align*}
		\sum_{\substack{l,h \in P \\ \textnormal{deg}(lh) \leq n}} \frac{1}{|l||h|} &= \sum_{\substack{l \in P \\ \textnormal{deg}(l) \leq n-1}} \frac{1}{|l|}\sum_{\substack{h \in P \\ \textnormal{deg}(h) \leq n - \text{deg}(l)}} \frac{1}{|h|}
		\\&=\sum_{\substack{l \in P \\ \textnormal{deg}(l) \leq n-1}} \frac{1}{|l|} \left( \log(n - \text{deg}(l)) + A_1 + O \left( \frac{1}{n - \text{deg}(l)} \right) \right) \\
		& = \sum_{\substack{l \in P \\ \textnormal{deg}(l) \leq n-1}} \frac{\log(n - \text{deg}(l))}{|l|} + A_1\log n+A_1^2 + O \left( \frac{\log n}{n} \right).
\end{align*}
Again by Lemma \ref{Lemma_Mertens}, we have 
\begin{align*}
	\sum_{\substack{l \in P \\ \textnormal{deg}(l) \leq n-1}} \frac{\log(n - \text{deg}(l))}{|l|}
	&=\log n\sum_{\substack{l \in P \\ \textnormal{deg}(l) \leq n-1}}\frac{1}{\left|l\right|}+\sum_{\substack{l \in P \\ \textnormal{deg}(l) \leq n-1}}\frac{\log\left(1-\frac{\deg (l)}{n}\right)}{\left|l\right|}
	\\&=\log^2 n+A_1\log n +\sum_{\substack{l \in P \\ \textnormal{deg}(l) \leq n-1}}\frac{\log\left(1-\frac{\deg (l)}{n}\right)}{\left|l\right|}
	+O\left(\frac{\log n}{n}\right).
\end{align*}
Since $\pi_{q}(m)=\frac{q^m}{m}+O\left(\frac{q^{m/2}}{m}\right)$, we have 
\begin{align*}
	\sum_{\substack{l \in P \\ \textnormal{deg}(l) \leq n-1}}\frac{\log\left(1-\frac{\deg (l)}{n}\right)}{\left|l\right|}&=\sum_{1\leqslant m \leqslant n-1}\pi_q(m)\frac{\log\left(1-\frac{m}{n}\right)}{q^m}
	\\&=\sum_{1\leqslant m \leqslant n-1}\frac{\log\left(1-\frac{m}{n}\right)}{m}+O\left(\frac{1}{n}\right).
\end{align*}
Either by using Riemann sums for the negative decreasing function $F(x):=\frac{\log(1-x)}{x}$ on $[0,1]$ and the fact that $\int_{0}^{1}F(x)\,dx=-\frac{\pi^2}{6}$, or by using the Taylor series of $\log\left(1-\frac{m}{n}\right)$ and the bound $\left|\sum_{m\leqslant r}m^j-\frac{r^{j+1}}{j+1} \right|\leqslant r^j$ for all natural numbers $j$ and $r$ (which can be proved by induction on $j$), one can show that 
\begin{align*}
		\sum_{1\leqslant m \leqslant n-1}\frac{\log\left(1-\frac{m}{n}\right)}{m}=-\frac{\pi^2}{6}+O\left(\frac{\log n}{n}\right).	
	\end{align*}
Combining the results above, we obtain the estimate in (\ref{S_1_result}) for $S_1(n)$.

Now, we consider $S_2(n)$. We show that 
\begin{align}\label{S_2_result}
		S_2(n) = L(2)\log n + A_1 L(2) - L(3) + O \left( \frac{1}{n} \right).
\end{align}
We have 
\begin{align*}
		S_2(n) 
		=\sum_{\substack{l,h \in P \\ \textnormal{deg}(l^2 h) \leq n}} \frac{1}{|l|^2|h|} - \sum_{\substack{l \in P \\ \textnormal{deg}(l) \leq n/3}} \frac{1}{|l|^3}
		= \sum_{\substack{l,h \in P \\ \textnormal{deg}(l^2 h) \leq n}} \frac{1}{|l|^2 |h|} - L(3) + O \left( \frac{q^{-2n/3}}{n} \right).
\end{align*}
Since by \lmaref{sumqk/2},
\begin{align*}
	\sum_{\substack{l,h \in P \\ \textnormal{deg}(l^2 h) \leq n}} \frac{1}{|l|^2 |h|}& = \sum_{\substack{h \in P \\ \textnormal{deg}(h) \leq n-2}} \frac{1}{|h|} \left( \sum_{\substack{l \in P \\ \textnormal{deg}(l) \leq \frac{n - \text{deg}(h)}{2}}} \frac{1}{|l|^2} \right)
	\\&=L(2) \sum_{\substack{h \in P \\ \textnormal{deg}(h) \leq n-2}} \frac{1}{|h|} + O \left( \sum_{\substack{h \in P \\ \textnormal{deg}(h) \leq n-2}} \frac{q^{-\frac{n - \text{deg}(h)}{2}}}{|h| (n - \text{deg}(h))} \right)
	\\&=L(2) (\log n + A_1) + O \left( \frac{1}{n} \right)+O \left( \sum_{m=1}^{n-1} \frac{q^{-\frac{m}{2}}}{m (n-m)} - \frac{q^{-1/2}}{n-1} \right) 
	\\&=L(2) (\log n + A_1) + O \left( \frac{1}{n} \right),
\end{align*}
we obtain the estimate in (\ref{S_2_result}) for $S_2(n)$.

Now we consider $S_3(n)$ for which we show that
\begin{align}\label{S_3_result}
	S_3(n) = L(2)^2 - L(4) + O  \left( \frac{1}{n} \right).
\end{align}
We have
\begin{align*}
		S_3(n) 
		=\sum_{\substack{l,h \in P \\ \textnormal{deg}(l^2 h^2) \leq n}} \frac{1}{|l|^2 |h|^2} - \sum_{\substack{l \in P \\ \textnormal{deg}(l) \leq n/4}} \frac{1}{|l|^4}
		= \sum_{\substack{l,h \in P \\ \textnormal{deg}(l^2 h^2) \leq n}} \frac{1}{|l|^2 |h|^2} - L(4) + O \left( \frac{q^{-3n/4}}{n} \right).
\end{align*}
Since
\begin{align*}
	\sum_{\substack{l,h \in P \\ \textnormal{deg}(l^2 h^2) \leq n}} \frac{1}{|l|^2 |h|^2} &= \sum_{\substack{l \in P \\ \textnormal{deg}(l) \leq \frac{n-2}{2}}} \frac{1}{|l|^2} \left( \sum_{\substack{h \in P \\ \textnormal{deg}(h) \leq \frac{n - 2 \text{deg}(l)}{2}}} \frac{1}{|h|^2} \right) \\
	&=L(2) \left( L(2) + O \left( \frac{q^{-(n-2)/2}}{n} \right) \right) + O \left( \frac{1}{n} \right)
	\\&=L(2)^2+O\left(\frac{1}{n}\right),
\end{align*}
we obtain the estimate in (\ref{S_3_result}) for $S_3(n)$. Hence, by Theorem \ref{mainresult_first_moment}, (\ref{Expanding_square}), (\ref{S_1_result}), (\ref{S_2_result}), and (\ref{S_3_result}), we obtain the first assertion in Theorem \ref{mainresult_second_moment} about the second moment of $\omega_1$.

Now we consider the second moment of $\omega_k$ for $k\geqslant 2$. We have 
\begin{align}\label{expanding_square_k>1}
		\nonumber \sum_{f \in M_n} \omega_k^2(f)&=\sum_{f \in M_n} \omega_k(f) + \sum_{f \in M_n} \left( \sum_{\substack{l,h \in P, \ l \neq h \\ l^k || f, \ h^k || f}} 1 \right)
		\\\nonumber&=\sum_{f \in M_n} \omega_k(f)+ \sum_{\substack{l,h \in P, \ l \neq h \\ \text{deg}(lh) \leq n/k}} \left( \sum_{\substack{f \in M_n \\ (lh)^k | f}} 1 -  \sum_{\substack{f \in M_n \\ l^{k+1}h^k | f}} 1 - \sum_{\substack{f \in M_n \\ l^k h^{k+1} | f}} 1 + \sum_{\substack{f \in M_n \\ (lh)^{k+1} | f}} 1 \right).
		\\&=\sum_{f \in M_n} \omega_k(f)+q^n \sum_{\substack{l,h \in P, \ l \neq h \\ \text{deg}(lh) \leq n/k}} \left( \frac{1}{|l|^k} - \frac{1}{|l|^{k+1}} \right) \left( \frac{1}{|h|^k} - \frac{1}{|h|^{k+1}} \right).
\end{align}
We have 
\begin{align*}
	& \sum_{\substack{l,h \in P, \ l \neq h \\ \text{deg}(lh) \leq n/k}} \left( \frac{1}{|l|^k} - \frac{1}{|l|^{k+1}} \right) \left( \frac{1}{|h|^k} - \frac{1}{|h|^{k+1}} \right) \\
	& = \sum_{\substack{l,h \in P \\ \text{deg}(lh) \leq n/k}} \left( \frac{1}{|l|^k} - \frac{1}{|l|^{k+1}} \right) \left( \frac{1}{|h|^k} - \frac{1}{|h|^{k+1}} \right)
	 - \sum_{\substack{l \in P \\ \text{deg}(l) \leq n/2k}} \left( \frac{1}{|l|^k} - \frac{1}{|l|^{k+1}} \right)^2
\end{align*}
and 
\begin{align*}
	\sum_{\substack{l \in P \\ \text{deg}(l) \leq n/2k}} \left( \frac{1}{|l|^k} - \frac{1}{|l|^{k+1}} \right)^2 &= \sum_{l \in P} \left( \frac{1}{|l|^k} - \frac{1}{|l|^{k+1}} \right)^2 - \sum_{\substack{l \in P \\ \text{deg}(l) > n/2k}} \left( \frac{1}{|l|^k} - \frac{1}{|l|^{k+1}} \right)^2
		\\&= L(2k) + L(2k+2) - 2L(2k+1)
+O \left( \frac{q^{-\frac{n}{2k} (2k-1)}}{n}\right).
\end{align*}
Moreover, we have 
\begin{align*}
	&\sum_{\substack{l,h \in P \\ \text{deg}(lh) \leq n/k}} \left( \frac{1}{|l|^k} - \frac{1}{|l|^{k+1}} \right) \left( \frac{1}{|h|^k} - \frac{1}{|h|^{k+1}} \right) \notag 
	\\& = \sum_{\substack{l \in P \\ \text{deg}(l) \leq (n/k) -1}} \left( \frac{1}{|l|^k} - \frac{1}{|l|^{k+1}} \right) \sum_{\substack{h \in P \\ \text{deg}(h) \leq (n/k) - \text{deg}(l)}} \left( \frac{1}{|h|^k} - \frac{1}{|h|^{k+1}} \right) \notag \\
	& = \sum_{\substack{l \in P \\ \text{deg}(l) \leq (n/k) -1}} \left( \frac{1}{|l|^k} - \frac{1}{|l|^{k+1}} \right) \left( L(k) - L(k+1) + O \left( \frac{q^{-(k-1) \left(\frac{n}{k} - \text{deg}(l) \right)}}{(n/k) - \text{deg}(l)} \right) \right).
\end{align*}
The contribution of the error term above is bounded by 
\begin{align*}
		& \ll q^{-\frac{n(k-1)}{k}} \sum_{\substack{l \in P \\ \text{deg}(l) \leq \lfloor n/k \rfloor -1}} \frac{1}{|l| \left( \lfloor n/k \rfloor - \text{deg}(l) \right)}  \notag \\
	& \ll \frac{\log \lfloor n/k \rfloor}{\lfloor n/k \rfloor \cdot q^{\frac{n}{k}(k-1)}}  \\
	& \ll \frac{\log n}{n \cdot q^{\frac{n}{k}(k-1)}}.
\end{align*}
Since 
\begin{align*}
	\sum_{\substack{l,h \in P \\ \text{deg}(lh) \leq (n/k)-1}} \left( \frac{1}{|l|^k} - \frac{1}{|l|^{k+1}} \right) = L(k) - L(k+1) + O \left( \frac{1}{n \cdot q^{\frac{n}{k}(k-1)}} \right),
\end{align*}
we have 
\begin{align*}
		\sum_{\substack{l,h \in P \\ \text{deg}(lh) \leq n/k}} \left( \frac{1}{|l|^k} - \frac{1}{|l|^{k+1}} \right) \left( \frac{1}{|h|^k} - \frac{1}{|h|^{k+1}} \right) = (L(k) - L(k+1))^2 +  O \left( \frac{\log n}{n \cdot q^{\frac{n}{k}(k-1)}} \right).
\end{align*}
Hence, by Theorem \ref{mainresult_first_moment}, (\ref{expanding_square_k>1}), and the estimates above, we obtain the second assertion in Theorem \ref{mainresult_second_moment}.
\qed
%

%
\section{Proof of Theorem \ref{mainresult_normal_order_k=1}}
We now show how \thmref{mainresult_normal_order_k=1} follows from \thmref{mainresult_first_moment} and \thmref{mainresult_second_moment}.
\begin{proof}
	Note that
	$$\sum_{f \in M_n} (\omega_1(f) - \log n)^2 = \sum_{f \in M_n} \omega_1^2(f) + \log^2 n\sum_{f \in M_n}1  - 2 \log n \sum_{f \in M_n} \omega_1(f).$$
	Applying \thmref{mainresult_first_moment} and \thmref{mainresult_second_moment} 
completes the first part of the proof. We thus have
	\begin{equation}\label{needed2}
		\sum_{f \in M_n} (\omega_1(f) - \log n)^2 \ll q^n \log n.
	\end{equation}
	Let $E_n$ be the set of monic polynomials $f$ of degree $n$ such that 
	$$\frac{|\omega_1(f) - \log n|}{\sqrt{\log n}} \geq (\log n)^{\epsilon'},\quad \quad (0<\epsilon'<1/2).$$
	Let $|E_n|$ be the cardinality of $E_n$. Thus we have
	$$\sum_{f \in M_n} (\omega_1(f) - \log n)^2 \geq \sum_{f \in E_n} (\omega_1(f) - \log n)^2 \geq \sum_{f \in E_n} (\log n)^{1+2\epsilon'} = (\log n)^{1+2\epsilon'} |E_n|$$
	which together with \eqref{needed2} yields
	$$|E_n| \ll \frac{q^n}{(\log n)^{2\epsilon'}} \ \text{ as } \ n \rightarrow \infty.$$
	This proves $|E_n| = o(q^n)$ as $n \rightarrow \infty$. Let $\epsilon > 0$ be any real number. Since $\epsilon' < 1/2$, there exist $n_0 \in \mathbb{N}$ such that $(\log n)^{-(\frac{1}{2} - \epsilon')} \leq \epsilon$ for all $n \geq n_0$. Hence as $n \rightarrow \infty$,
	the number of polynomials $f$ of degree $n$ satisfying
	$$\frac{|\omega_1(f) - \log n|}{\log n} \geq \epsilon $$
	is $o(q^n)$ which implies that $\omega_1(f)$ has  normal order $\log(\text{deg}(f))$.

\end{proof}

\section{Proof of Theorem \ref{mainresult_no_normal_order_k>1}}

\begin{proof}
First  assume that there exist a polynomial $f_0$ of degree $n_0$ such that $G(f_0) >0$. Then $G(f) > 0$ for all $f \in M_n$ with $n > n_0$. 
	For $n \in \mathbb{N}$, we define
	$$\mathcal{N}_0(n) = \{ f \in M_n \ | \ \omega_k(f) = 0 \}.$$ 
	Using $q \geq 2$, $|M_2|-\pi_q(2)= (q^2 + q)/2$, and $|M_3| - \pi_q(3) = (2q^3+q)/3$, we obtain
	\begin{align*}
		\sum_{l \in P} \frac{1}{|l|^2} 
		 & \leq \sum_{f \in M} \frac{1}{|f|^2} - 1 - q^{-4} \left( q^2 - \frac{q^2 - q}{2} \right) - q^{-6} \left( q^3 - \frac{q^3 - q}{3} \right) \notag \\
	& \leq \zeta_q(2) - 1 - \frac{q+1}{2q^3} - \frac{2q^2+1}{3q^5} \notag \\
	& = \frac{q}{q-1} - 1 - \frac{q+1}{2q^3} - \frac{2q^2+1}{3q^5} \notag \\
	& \leq \frac{23}{32}.
	\end{align*}
	Thus, we deduce
	\begin{align*}
	|\mathcal{N}_0(n)| = q^n - \sum_{\substack{f \in M_n \\ f \notin \mathcal{N}_0(n)}} 1 
	\geq q^n - \sum_{l \in P} \sum_{\substack{f \in M_n \\ l^k || f}} 1 
	\geq  q^n -  q^n \sum_{l \in P} \frac{1}{|l|^k} 
	\geq \frac{9}{32} q^n.
	\end{align*}
	Thus the number of monic polynomials $f$ of degree $n$ for which $G(f) >0$ and $\omega_k(f) = 0$ is not $o(q^n)$, and for such an $f$,
	\begin{equation}\label{omegaG}
	|\omega_k(f) - G(f)| > \frac{G(f)}{2}
	\end{equation}
	is satisfied. This proves that $\omega_k(f)$ does not have  normal order $G(f)$ when $G(f)$ is not identically 0.
	
	Next, we suppose $G(f) = 0$ for all $f \in M$. For $n \in \mathbb{N}$, we define
	$$\mathcal{N}_1(n) = \{ f \in M_n \ | \ \omega_k(f) = 1 \}.$$
	Let $h \in P_1$ denote a fixed monic irreducible of degree 1. We use this $h$ to obtain
	\begin{align*}
	|\mathcal{N}_1(n)| \geq \sum_{\substack{f \in M_n, \ \nu_{h}(f) = k \\ \nu_l(f) < k \ \forall l \in P \text{ with } l \neq h}} 1  
	\geq q^{n-k} \left( 1 - q^{-1} -  \sum_{\substack{l \in P \\ l \neq h}} \frac{1}{|l|^k} \right) 
	\geq \frac{1}{32} q^{n-k}.
	\end{align*}
	Thus the number of monic polynomials $f$ of degree $n$ for which $G(f) = 0$ and $\omega_k(f) = 1$ is not $o(q^n)$, and all such $f$'s again satisfy \eqref{omegaG}. 
	This completes the proof.
\end{proof}
\section{Proof of Theorem \ref{mainresult_Erdos_Kac_k=1}}

We begin this section by proving a second version of Erd\H{o}s-Kac theorem for function fields which is a consequence of the first version given in \eqref{kacerdosver1}.
\begin{thm}\label{ver2}(Erd\H{o}s-Kac theorem for function fields - Version II)
	Let $a \in \mathbb{R}$. Then
	$$\lim_{n \rightarrow \infty} \frac{1}{q^n} \left| \left\{ f \in M_n \ \Big| \ \frac{\omega(f) - \log n}{\sqrt{\log n}} \leq a \right\} \right| = \Phi(a)$$
	where $\Phi(a)$ is defined in \eqref{Nadef}. 
\end{thm}
\begin{proof}
	Let us denote the set
	$$\mathscr{Z}(n) := \left\{ f \ \Big| \ f \in M_m, \ m \leq n, \  \frac{\omega(f) - \log m}{\sqrt{\log m}} \leq a \right\}.$$
	Therefore
	\begin{align*}\label{eqreq1}
		\frac{|\mathscr{Z}(n)|}{| \{ f \ | \ f \in M_m, \ m \leq n \} |} & = \frac{|\mathscr{Z}(n-1)|}{| \{ f \ | \ f \in M_m, \ m \leq n-1 \} |} \frac{| \{ f \ | \ f \in M_m, \ m \leq n-1 \} |}{| \{ f \ | \ f \in M_m, \ m \leq n \} |} \notag \\
		& \quad + \frac{\left| \left\{ f \in M_n \ \Big| \  \frac{\omega(f) - \log n}{\sqrt{\log n}} \leq a \right\} \right|}{q^n} \frac{q^n}{| \{ f \ | \ f \in M_m, \ m \leq n \} |}.
	\end{align*}
	Taking the limits $n \rightarrow \infty$ on both sides and using \eqref{kacerdosver1} completes the proof.
\end{proof}
\begin{proof}[Proof of \thmref{mainresult_Erdos_Kac_k=1}.]
	For  $f \in M_n$ and a function $g: M \rightarrow \mathbb{R}_{\geq 0}$, we define
	$$r_g(f) := \frac{g(f)- \log n}{\sqrt{\log n}}$$
	and 
	$$D(g,n,a) := \frac{1}{q^n} | \{ f \in M_n \ | \ r_g(f) \leq a \} |.$$
	Note that $D(\omega,n,a) \leq D(\omega_1, n, a)$. Thus, by \thmref{ver2}, we obtain 
	\begin{equation}\label{erka1}
		\Phi(a) \leq \liminf_{n \rightarrow \infty} D(\omega_1,n,a).
	\end{equation}
	Let $\epsilon > 0$ be a real number. Let us denote the set 
	$$\mathcal{D}(n,\epsilon) = \left\{ f \in M_n \ \Big| \ \frac{\omega(f) - \omega_1(f)}{\sqrt{\log n}} \leq \epsilon \right\} .$$
	Notice that
	\begin{equation}\label{kaer2}
		D(\omega_1,n,a) \leq D(\omega, n, a + \epsilon) + \frac{1}{q^n} \left| M_n \backslash \mathcal{D}(n,\epsilon) \right|.
	\end{equation}
	By \eqref{uselate} and the calculation in the proof of \thmref{mainresult_first_moment}, we obtain
	\begin{align*}
		\sum_{f \in M_n \backslash \mathcal{D}(n,\epsilon)} (\omega(f) - \omega_1(f)) 
		& \leq \sum_{f \in M_n} (a_1(f) + a_2(f)) \\
		& = L(2) q^n + O \left( n \cdot q^{\frac{n}{2}} \right).
	\end{align*}
	Moreover, 
	we have 
	$$\sum_{f \in M_n \backslash \mathcal{D}(n,\epsilon)} (\omega(f) - \omega_1(f)) > \epsilon \sqrt{\log n} \ \left| M_n \backslash \mathcal{D}(n,\epsilon)  \right|$$
	which together with the previous inequality yields
	$$\left| M_n \backslash \mathcal{D}(n,\epsilon) \right| \ll \frac{q^n}{\epsilon \sqrt{\log n}}.$$
	Thus $\left| M_n \backslash \mathcal{D}(n,\epsilon) \right| = o(q^n)$ as $n \rightarrow \infty$. Using \eqref{kaer2} and \thmref{ver2} with $a+\epsilon$ instead of $a$, we deduce that
	$$ \limsup_{n \rightarrow \infty} D(\omega_1,n,a) \leq \Phi(a + \epsilon).$$
	Since $\epsilon$ is arbitrary, combining the above inequality with \eqref{erka1} completes the proof.
\end{proof}

\section{Acknowledgments}
The authors would like to thank the referees for providing valuable comments to improve the paper.

\bibliographystyle{plain} 

\Addresses

\end{document}